\documentclass[twoside, 12pt, reqno]{amsart}

\usepackage[left=2.5cm, right=2.5cm, top=3cm, bottom=2.5cm]{geometry}

\usepackage[colorlinks=true, pdfstartview=FitV, linkcolor=blue,citecolor=blue, urlcolor=blue]{hyperref}

\usepackage[usenames]{xcolor}
\definecolor{labelkey}{rgb}{0,0,1}
\definecolor{Red}{rgb}{0.7,0,0.1}
\definecolor{Green}{rgb}{0,0.7,0}

\usepackage{amsfonts, amssymb, amsmath, amsthm, mathrsfs, bbm, cjhebrew, gensymb, textcomp, mathtools, dsfont,tikz}

\usepackage[normalem]{ulem}

\usepackage{commath, setspace, subcaption, parcolumns, multirow, multicol, accents, comment, marginnote, verbatim, empheq, enumerate, stackrel, enumitem}

\usepackage{cleveref}
\usepackage{graphicx}
\usepackage{epsfig}
\usepackage{psfrag}
\usepackage{float}
\usepackage[active]{srcltx}
\usepackage[pagewise, mathlines]{lineno}

\newtheorem{remark}{Remark}
\newtheorem{lemma}{Lemma}
\numberwithin{lemma}{section}

\newtheorem{theorem}{Theorem}
\numberwithin{theorem}{section}

\numberwithin{proposition}{section}

\numberwithin{equation}{section}

\newcommand{\al}{\alpha}
\newcommand{\be}{\beta}

\newcommand{\eps}{\epsilon}
\newcommand{\veps}{\varepsilon}

\newcommand{\kap}{\kappa}

\newcommand{\Lam}{\Lambda}
\newcommand{\si}{\sigma}

\newcommand{\tht}{\theta}

\newcommand{\RR}{\mathbb{R}}

\newcommand{\lb}{\big\langle}
\newcommand{\rb}{\big\rangle}

\newcommand{\Sob}[2]{\lVert#1\rVert_{#2}}

\newcommand{\bdy}{\partial}

\newcommand{\til}[1]{\widetilde{#1}}

\newcommand{\Hdot}{\dot{H}}

\DeclareMathOperator*{\esssup}{ess\,sup}
\DeclareMathOperator*{\Div}{div}
\makeatother

\title[On wellposedness and limiting behavior of generalized SQG equations]{On wellposedness and limiting behavior of generalized SQG equations}
\author{Anuj Kumar$^{1,\dagger}$}
\address{$^1$Department of Mathematics,
Indian Institute of Technology Jodhpur, India}
\address{$\dagger$ corresponding author}
\email[A Kumar]{anujkumar@iitj.ac.in}
\date{\today}
\thanks{}
\begin{document}
\begin{abstract}
	We consider the two-dimensional generalized surface quasi-geostrophic equations in $\mathbb{R}^2$, given by
    \[\partial_t \theta+u\cdot \nabla \theta=0,\quad u=-\nabla^{\perp}\Lambda^{\beta-2}\theta,\,\beta\in [1,2).\] When $\beta=1$, the equation defines the SQG equation and for $\beta>1$, it defines a family of more singular active scalar equations. We prove that if the interval of existence of the smooth solution to the generalized SQG equations for some $\beta_0\in[1,2)$ is $[0,T]$, then with the same initial data, the interval of existence of the generalized SQG equations for $\beta$ close to $\beta_0$ also contains $[0,T]$. To prove these results, we develop estimates for a conservation law with flux modified around the generalized SQG equations. Furthermore, we also prove some new commutator estimates with bounds uniformly bounded as $\beta\to 1$ that may be of independent interest.
\end{abstract}
\maketitle
{\noindent \small {\it {\bf Keywords: Generalized surface quasi-geostrophic (gSQG) equation, smooth solution, Sobolev spaces}
}  }
\section{Introduction}
The main equation of interest in this paper is the two-dimensional (2D) generalized surface quasi-geostrophic (gSQG) equations given by
\begin{align}\label{gSQG}
    \begin{cases}
        &\bdy_t \tht+u\cdot \nabla \tht=0,\\
       &u=\nabla^{\perp}\psi=\left(-\bdy_{x_2}\psi,\bdy_{x_1}\psi\right),\\
        &\Delta \psi=\Lam^{\be}\tht,\quad x\in \mathbb{R}^2,\,t>0.
    \end{cases}
\end{align}
In the above equation, $\tht(x,t)$ denotes the evolving scalar, $u(x,t)$ denotes the advecting velocity field, and $\psi(x,t)$ its corresponding stream function. The operator $\Lam$ denotes the fractional Laplacian operator $(-\Delta)^{1/2}$, with Fourier multiplier $|\xi|$. The parameter $\beta\in (0,2)$ quantifies the singularity in the constitutive laws. We assume the domain to be $\RR^2$, and consider the Cauchy problem with the initial data $\tht(x,0)=\tht_0$. \par
The family of equations in \eqref{gSQG} for the range $\be \in (0,1]$ was introduced in \cite{ChaeConstantinWu2011}, while the case of $\be\in (1,2)$ was first studied in \cite{ChaeConstantinCordobaGancedoWu2012}. 
For $\be \in [0,1]$, \eqref{gSQG} interpolate between the 2D Euler equation in vorticity form ($\be=0$) and the surface quasi-geostrophic equation (SQG) ($\be=1$). On the other hand, for $\be \in (1,2)$, \eqref{gSQG} represents a family of active scalar equations with increasingly more singular constitutive laws than the SQG equation.\par
In this paper, we consider \eqref{gSQG} in the regime defined by $\be\in [1,2)$ and study the behavior of solutions as $\be \to \be_0$, where $\be_0$ is allowed to be any value in $[1,2)$. More specifically, we establish the fact that the interval of existence of smooth solutions to \eqref{gSQG} is non-decreasing under small perturbations of the parameter $\be$. The main results are stated in \cref{main:T:1} for the case of $\be_0 \in(1,2)$ and in \cref{main:T:2} for the case of $\be_0=1$. Note that $\be_0=1$ case represents the study of solutions for small perturbations around the SQG equation as $\be \to 1^{+}$. The corresponding results for $\be\to 1^{-}$ were established in \cite{YuZhengJiu2019}. Therefore, the result obtained here, along with that of \cite{YuZhengJiu2019} provides a complete understanding of the stability behavior of the SQG equation under small perturbations in the parameter $\be$.\par
The case of the SQG equation ($\be=1$) describes the evolution of surface temperature or buoyancy in a rapidly rotating, stably stratified fluid with potential vorticity. It is a fundamental equation in geophysics and meteorology and has attracted a lot of attention in the last three decades, especially due to properties such as vortex stretching that are similar to those of the 3D Euler equation in vorticity form \cite{ConstantinMajdaTabak1994,Cordoba1998}. Additionally, it also exhibits features of 2D turbulence similar to those of the 2D Euler equation \cite{MajdaTabak1996}. From a purely mathematical perspective, the SQG equation is also an important toy model for studying well-posedness related questions. Depending on the balance between the dissipative term of the form $\Lam^\al \tht$ and the nonlinear term, the dissipative SQG equation is usually divided into three regimes: subcritical ($\al>1$), critical ($\al=1$), and supercritical ($\al<1$). Despite being a 2D model, the question of global regularity for the dissipative SQG equation remains an outstanding open problem in the supercritical regime even though global regularity has been established in the subcritical \cite {ConstantinWu1999,Resnick1995} and critical regimes. In particular, the global regularity in the case of critical dissipation $\al=1$ remained open for a long time until it was settled independently by several sophisticated methods \cite{CaffarelliVasseur2010,ConstantinVicol2012,KiselevNazarovVolberg2007}. In the supercritical regime, several results on local well-posedness for large data, global well-posedness for small data, and Gevrey class smoothing of solutions in time have been established \cite{Miura2006, ChenMiaoZhang2007,HmidiKeraani2007, Wu2004, Biswas2014, BiswasMartinezSilva2015, Dong2010}.\par
The Cauchy problem for the regime $\be\in (1,2)$ was first addressed in \cite{ChaeConstantinCordobaGancedoWu2012}, where local well-posedness was established in $H^4$. Subsequently, this result was sharpened to show that it suffices to have the initial data in $H^{\be+1+\eps}$ for recovering local well-posedness \cite{HuKukavicaZiane2015, Li2019}. In the presence of fractional dissipation of the form $\Lam^\al \tht$, local well-posedness in the critical space $H^{\be+1-\al}$ as well as Gevrey class regularity of solutions were established by the present author with M.S. Jolly and V.R. Martinez \cite{JollyKumarMartinez2020a}. The central idea of the proof was the use of an adapted approximation scheme with flux suitably modified to exploit the cancellation properties of \eqref{gSQG}, which appear to be instrumental in controlling the nonlinearity, as observed in previous studies on \eqref{gSQG}   \cite{ChaeConstantinCordobaGancedoWu2012, HuKukavicaZiane2015}. In this manuscript, we use a similar approach of developing estimates for an equation of the form of a conservation law with flux modified around the gSQG equations (see \eqref{conservation:law}). One can think of it as a perturbation of the gSQG equations so that the nonlinearity appears in the form of a commutator, thus affording inherent cancellation properties associated with \eqref{gSQG}. \par
Several recent studies have investigated the question of well-posedness for \eqref{gSQG} in the borderline space  $H^{\be+1}$. This question for the Euler endpoint $\be=0$ was settled in the seminal work of J. Bourgain and D. Li\cite{BourgainLi2015}, where a mechanism for norm inflation was developed to prove ill-posedness in $H^1$ (see also \cite{ElgindiMasmoudi2020,Kwon2020}). Recently, in a series of works, ill-posedness has also been established in the range $\be\in (0,2)$ by D. C$\acute{\text{o}}$rdoba and L. Mart$\acute{\text{i}}$nez-Zoroa \cite{CordobaZoroa-Martinez2021, CordobaZoroa-Martinez2022} and I.-J. Jeong and J. Kim. Complementary results on the impossibility of uniform continuity of the solution operator have also been obtained for Euler in \cite{BourgainLi2019,HimonasMisiolek2010,Inci2015,MisiolekYoneda2016,MisiolekYoneda2018}, SQG in \cite{Inci2018}, and gSQG in \cite{MisiolekVu2024}. In contrast to these results, some studies have investigated the possibility of regularizing the velocity just enough to recover well-posedness in $H^{\be+1}$. D. Chae and J. Wu showed that a mild regularization using a negative power of a logarithmic multiplier suffices when $\be \in[0,1]$. The complementary result when $\be\in(1,2)$ was subsequently obtained in \cite{JollyKumarMartinez2020b}. \par
Henceforth, we will use the notation that $\{u^\be,\tht^\be\}$ and $\{u^{\be_0
},\tht^{\be_0}\}$ denote the solutions of \eqref{gSQG} with constitutive law parameters $\be$ and $\be_0$ respectively.
Let $\be_0 \in [1,2)$ and $T>0$. In this paper, we aim to show that if $\tht^{\be_0}$, defines a unique smooth solution on $[0,T]$, then there exists $\delta>0$ such that for any $\be \in (1,2)$ satisfying $0<|\be-\be_0|<\delta$, the Cauchy problem for \eqref{gSQG} with the same initial data and constitutive law parameter $\be$ also possesses a unique smooth solution defined on $[0,T]$. This study is motivated by the results in \cite{YuZhengJiu2019} and \cite{Constantin1986Euler}. In \cite{YuZhengJiu2019}, the corresponding results for the case $\be_0\in [0,1]$ and $\be \in (0,1)$ were established, while in \cite{Constantin1986Euler}, it was shown that if the Cauchy problem for the 3D incompressible Euler equation has a unique smooth solution on some interval $[0,T]$, then so does the corresponding 3D incompressible Navier-Stokes equation with the same initial data if the viscosity is assumed sufficiently small. \par
Let us denote by 
\[\overline{u}=u^\be-u^{\be_0}\quad \text{and}\quad\overline{\tht}=\tht^{\be}-\tht^{\be_0}.\] Then, as observed in \cite{YuZhengJiu2019}, the main difficulty of the proof lies in determining how to use the information of $\Sob{\overline{\tht}(t)}{H^s}$ to control $\overline{u}$ in the energy arguments. 
In \cite{YuZhengJiu2019}, the authors considered the decomposition of $\overline{u}$ as
\begin{align*}
\overline{u}&=\nabla^\perp \Lam^{\be-2}\overline{\tht}+(\nabla^\perp \Lam^{\be-2}-\nabla^\perp \Lam^{\be-2})\tht^{\be_0}\\
&=\overline{u}_I+\overline{u}_{II}.
\end{align*}
They then estimated $\overline{u}_I$ directly using Hardy-Littlewood-Sobolev inequality when $\be_0 \in (0,1)$ and established new product estimates for when $\be_0=0,1$. This recourse is no longer available to us because the constitutive laws are more singular when $\be_0 \in (1,2)$ or when $\be_0=1$ and $\be \to 1^{+}$. Therefore, one must use a more delicate approach which consists of decomposing the nonlinear terms in terms of commutators that provide the requisite control. In addition, when $\be \to 1^{+}$, we need to develop new bounds (see \cref{commutator:estimate:1}) that remain valid in the limit, albeit at some cost of additional regularity of the initial data (see \cref{main:T:2}). We now state our main results. Note that $\tht^\be, \tht^{\be_0}$ denote the solutions to \eqref{gSQG} with constitutive law parameters $\be$ and $\be_0$, respectively.
\begin{theorem}\label{main:T:1}
    Let $\be_0 \in (1,2)$. Let $\tht^{\be_0}(x,t)$ denote the solution to \eqref{gSQG} with initial data $\tht_0 \in H^{s+1} (s>3)$, valid on some interval $[0,T]$. Then, there exists $\delta>0$ depending on $T$ and $\int_0^T \Sob{\tht^{\be_0}}{H^{s+1}}\,dt$ such that if $\be \in (1,2)$ and $|\be-\be_0|< \delta$, the solution $\tht^{\be}(x,t)$ to \eqref{gSQG} with the same initial data is also smooth on $[0,T]$. Moreover, it holds that
    \begin{align*}
        \Sob{\tht^{\be}(t)-\tht^{\be_0}(t)}{H^s}\le C(|\be-\be_0|^{2-\be}+|\be-\be_0|^{2-\be_0}+|\be-\be_0||\ln|\be-\be_0||),
    \end{align*}
    where $C$ is a positive constant depending on $T$ and $\int_0^T\Sob{\tht^{\be_0}}{H^{s+1}}\,dt$.
\end{theorem}
For the end point case $\be_0=1$, our result is as follows:
\begin{theorem}\label{main:T:2}
    Let $1=\be_0<\be<2$. Let $\tht^{\be_0}(x,t)$ denote the solution to \eqref{gSQG} with  initial data $\tht_0 \in H^{s+1}\cap L^1 (s>3)$, valid on some interval $[0,T]$. Then, there exists $\delta>0$ depending on $T$ and $\int_0^T \Sob{\tht^{\be_0}}{H^{s+1}\cap L^1}\,dt$ such that if $0<\be-\be_0\le \delta$, the solution $\tht^{\be}(x,t)$ to \eqref{gSQG} with the same initial data is also smooth on $[0,T]$. Moreover, it holds that
    \begin{align*}
      \Sob{\tht^{\be}(t)-\tht^{\be_0}(t)}{H^s}\le C((\be-1)^{2-\be}+(\be-1)|\ln |\be-1||),
    \end{align*}
    where $C$ is a positive constant depending on $T$ and $\int_0^T\Sob{\tht^{\be_0}}{H^{s+1}\cap L^1}\,dt$.
\end{theorem}
\section{Mathematical preliminaries}
Denote by $\mathcal{S}(\RR^n)$ the space of Schwartz class functions on $\RR^n$ and by $\mathcal{S}'(\RR^n)$ the space of tempered distributions on $\RR^n$. For $p\in [1,\infty]$, we denote by $L^p(\RR^n)$ the space of Lebesgue integrable functions of order $p$ on $\RR^n$. The norm in the space $L^p(\RR^n)$ is defined by
\begin{align*}
    \Sob{f}{L^p(\RR^n)}:=\begin{cases}
        \left(\int_{\RR^n}|f(x)|^p\,dx\right)^{\frac{1}{p}},\quad p\in[1,\infty),\\
        \esssup_{x\in \RR^n}|f(x)|,\quad p=\infty.
    \end{cases}
\end{align*}
Recall that, with this norm, $L^p(\RR^n)$ defines a Banach space. In the particular case $p=2$, it defines a Hilbert space as it can be endowed with the inner product given by
\begin{align*}
    \lb f,g\rb:=\int_{\RR^n}f(x)\overline{g(x)}\,dx.
\end{align*}
Given $f\in \mathcal{S}'(\RR^n)$, we denote by $\widehat{f}$, the Fourier transform of $f$. We denote by
\[\Lam^\si=(-\Delta)^{\frac{\si}{2}},\quad J^\si=(I-\Delta)^{\frac{\si}{2}},\]
where
\[\widehat{\Lam^\si f}(\xi)=|\xi|^\si \widehat{f}(\xi),\quad\text{and}\quad \widehat{J^\si f}(\xi)=(1+|\xi|^2)^{\frac{\si}{2}}\widehat{f}(\xi),\]
for $\si\in \RR$.
For $\si \in \RR$ and $p\in [1,\infty]$, the homogeneous and the inhomogeneous Sobolev spaces on $\RR^n$ are defined as
\begin{align*}
    \dot{W}^{\si,p}(\RR^n):=\left\{f\in \mathcal{S}'(\RR^n):\, \Sob{f}{\dot{W}^{\si,p}(\RR^n)}:=\Sob{\Lam^\si f}{L^p(\RR^n)}<\infty\right\},\\
    W^{\si,p}(\RR^n):=\left\{f\in \mathcal{S}'(\RR^n):\, \Sob{f}{{W}^{\si,p}(\RR^n)}:=\Sob{J^\si f}{L^p(\RR^n)}<\infty\right\}.
\end{align*}
Since for $\si \ge 0$, we have
\[2^{-\si}\{1+|\xi|^2\}^{\si}\le 1+|\xi|^{2\si}\le 2^\si\{1+|\xi|^2\}^\si,\quad \forall \xi\in \mathbb{R}^2.\]
Hence, we will use the following equivalent definition of the inhomogeneous Sobolev spaces $H^{\si}$ when $\si\ge0$:
\begin{align}
    H^\si(\mathbb{R}^2){:=}\left\{f\in \mathcal{S}'(\RR^2): \Sob{f}{H^\si}^2{:=}\Sob{f}{L^2}^2+\Sob{f}{\Hdot^{\si}}^2<\infty\right\}.\label{def:inhom:Sob:norm}
\end{align}
We recall the Sobolev embedding theorem (see \cite{BahouriCheminDanchinBook2011}).
\begin{lemma}\label{lem:Sob:embedding}
    Let $p_2\neq \infty$. Then, we have $W^{\si_1,p_1}\hookrightarrow W^{\si_2,p_2}$ if and only if
    \[\si_1-\frac{n}{p_1}\ge \si_2-\frac{n}{p_2},\quad \frac{1}{p_1}\ge \frac{1}{p_2}.\]
\end{lemma}
Next, we recall Leibnitz product rule and Kato-Ponce commutator estimates \cite{KatoPonce1988, Li2019}.
\begin{lemma}\label{lem:KP}
    Let $\si>0$ and $1<p<\infty$. Let $p_1,p_2, p_3, p_4 \in (1,\infty]$ be such that
    \[\frac{1}{p}=\frac{1}{p_1}+\frac{1}{p_2}=\frac{1}{p_3}+\frac{1}{p_4}.\]Then
    \begin{align}
        \Sob{J^\si (fg)}{L^p(\RR^n)}&\le C\left(\Sob{f}{W^{\si, p_1}(\RR^n)}\Sob{g}{L^{p_2}(\RR^n)}+\Sob{g}{W^{\si,p_3}(\RR^n)}\Sob{f}{L^{p_4}(\RR^n)}\right),\label{est:Leibnitz}\\
        \Sob{J^\si(fg)-fJ^\si g}{L^p(\RR^n)}&\le C\left(\Sob{f}{W^{\si, p_1}(\RR^n)}\Sob{g}{L^{p_2}(\RR^n)}+\Sob{g}{W^{\si-1,p_3}(\RR^n)}\Sob{\nabla f}{L^{p_4}(\RR^n)}\right)\label{est:KP},
    \end{align}
    where $C$ are constants depending on $\si,n, p_1,p_2,p_3,$ and $p_4$. Furthermore, the above inequalities also hold when $J^\si$ is replaced by $\Lam^\si$.
\end{lemma}
We also require the following variant of \eqref{est:KP} which remains valid when $\Lam^\si$ is replaced by $\partial \Lam^{\si-1}$ \cite{Li2019}. For all $\si>0$ and $p\in (1,\infty)$, we have 
\begin{align}\label{est:KP:A}
\Sob{\Lam^\si (fg)-f\Lam^{\si}g}{L^p(\RR^n)}\le C\left(\Sob{\Lam^\si f}{L^p(\RR^n)}\Sob{g}{L^\infty(\RR^n)}+\Sob{\Lam^{\si-1}g}{L^p(\RR^n)}\Sob{\nabla f}{L^\infty(\RR^n)}\right),
\end{align}
where $C$ is a constant depending on $\si,p$.\par
We recall that for $\alpha \in (0,n)$, the operator $\Lam^{-\al}$ can also be defined by the following singular integral:
\begin{align*}
    \Lam^{-\al}f(x)=C_{\al}\int_{\RR^n}\frac{f(y)}{|x-y|^{n-\al}}\,dy,
\end{align*}
where 
\[C_{\al}=\pi^{-n/2}2^{-\al}\frac{\Gamma(n/2-\al/2)}{\Gamma(\al/2)}.\]
Finally, we recall the following elementary estimate from \cite{Constantin1986Euler} (see also \cite{YuZhengJiu2019}).
\begin{lemma}\label{lem:Gronwall:variant}
    Suppose $F(t)$ is a nonnegative continuous function on $[0,T]$ and $G$ is a positive constant. If $0<\nu\le \nu_0=(8TG\int_0^T F(t)\,dt)^{-1}$, then every nonnegative solution $y(t)$ of 
    \begin{align*}
        \begin{cases}
            \frac{dy(t)}{dt}&\le \nu F(t)+Gy(t)^2,\\
            y(0)&=0,
        \end{cases}
    \end{align*} 
    satisfies 
    \[y(t)\le 12\nu \int_0^T F(t)\,dt,\quad \forall t\in [0,T].\]
\end{lemma}
\begin{remark}
    For the rest of the manuscript, we adopt the convention that $C$ denotes a positive constant that may vary from line to line. Any dependencies on parameters may be specified in statements of lemmas, theorems, or when they are relevant.
\end{remark}
\section{Commutator estimates}
For the generalized SQG equations \eqref{gSQG}, the evolving scalar $\tht(x,t)$ is expressed in terms of the velocity $u(x,t)$ by a singular integral operator determined by the constitutive law. The singularity of this operator is quantified by $\be$, which, for $\be\in (1,2)$ defines a more singular relationship than that for the SQG equation. Due to this inherent difficulty, in order to obtain suitable control of the nonlinearity, one must exploit observations made in \cite{ChaeConstantinCordobaGancedoWu2012} and \cite{HuKukavicaZiane2015} in which the nonlinear terms were decomposed in terms of commutators, which allowed an effective allocation of derivatives. This process is carried out by invoking estimates on commutators as they appear in the energy arguments. We recall the following estimate from \cite{Li2019} (see also \cite{JollyKumarMartinez2020b}) that is employed frequently throughout the paper.
\begin{lemma}\label{lem:Li2019} For $\be\in (1,2)$, $p \in (1,\infty)$, and $f,g \in \mathcal{S}(\RR^2)$, we have
\begin{align}
    \Sob{[\nabla \Lam^{\be-2}\cdot,g]f}{L^p(\RR^2)}&\le C_1(\be,p)\Sob{f}{L^p(\RR^2)}\Sob{\Lam^{\be-1} g}{BMO(\RR^2)}\notag\\&\le C_1(\be,p)\Sob{f}{L^p(\RR^2)}\Sob{\Lam^{\be-1} g}{L^\infty(\RR^2)},
\end{align}
where $C_1$ is a positive constant.
\end{lemma}
In \cref{lem:Li2019}, the constant $C_1$ becomes unbounded as $\be\to 1$. As a result, this bound fails to provide the required control of the nonlinear terms in the proof of \cref{main:T:2}. To overcome this difficulty, we establish a new technical result which remains valid when $\be \to 1$ and provides the required control of the nonlinear terms as they appear in the energy arguments. 
\begin{lemma}\label{commutator:estimate:1}
    Let $\be \in (1,2)$ and $p\in (1,\infty)$. Let $q_1, r_1, q_2,r_2 
    \in [1,\infty]$  
    be such that \[\frac{1}{q_1}+\frac{1}{r_1}=\frac{1}{p}, \] and
    \[\frac{1}{q_2}+\frac{1}{r_2}>\frac{1}{p}\]
    Then, we have
    \begin{align}
        \Sob{[\nabla \Lam^{\be-2}\cdot,g]f}{L^p}&\le C_2(\be)\Sob{\nabla g}{L^{q_1}}\Sob{f}{L^{r_1}}+C_3(\be,p)\Sob{g}{L^\infty}\Sob{f}{L^{1}}. \label{commutator:1:est:1}\\
        \Sob{[\nabla \Lam^{\be-2}\cdot,g]f}{L^p}&\le C_2(\be)\Sob{\nabla g}{L^{q_1}}\Sob{f}{L^{r_1}}+C_4(\be)\Sob{g}{L^\infty}\Sob{f}{\dot{W}^{-1,p}}\notag\\&+C_5(\be,p,q_2,r_2)\Sob{\nabla g}{L^{q_2}}\Sob{f}{\dot{W}^{-1,r_2}}\label{commutator:1:est:2},
    \end{align}
    where $C_2$, $C_3$, $C_4$, and $C_5$ are positive constants that remain uniformly bounded as $\be \to 1$.
\end{lemma}
 \begin{proof}
    Let $\rho(s)\in C^{\infty}_c(\RR^2)$ be a smooth cut-off function satisfying
    \begin{align*}
        \rho(s)=\begin{cases}
            1,\quad |s|\le 1,\\
            0,\quad |s|\ge 2,
        \end{cases}
    \end{align*}
and such that $|\rho'(s)|\le 2$. By definition
    \begin{align*}
        ([\nabla \Lam^{\be-2}\cdot,g]f)(x)&=\text{P.V.\,} C_\be\int_{\RR^2}\frac{(x-y)}{|x-y|^{2+\be}}(g(y)f(y)-g(x)f(y))\,dy.
        \end{align*}
        For convenience, we will take $C_\be=1$. We have
        \begin{align*}
        ([\nabla \Lam^{\be-2}\cdot,g]f)(x)&=\int_{\RR^2}\frac{(x-y)}{|x-y|^{2+\be}}\rho(|x-y|)(g(y)-g(x))f(y)\,dy\\&+\int_{\RR^2}\frac{(x-y)}{|x-y|^{2+\be}}(1-\rho(|x-y|))(g(y)-g(x))f(y)\,dy\\&=I+II.
    \end{align*}
    Using the mean value theorem and making a change of variables, we have
    \begin{align*}
        |I|&\le \int_{|x-y|\le 2}\frac{1}{|x-y|^{\be}}\left(\int_0^1 |\nabla g(x-\tau (x-y))|\,d\tau \right)|f(y)|\,dy\\
        &=\int_0^1 \int_{|z|\le 2}\frac{1}{|z|^{\be}}|\nabla g(x-\tau z)||f(x-z)|\,dy\,d\tau.
    \end{align*}
    Applying Minkowski inequality and H\"older's inequality, we obtain
    \begin{align}\label{est:commutator:I}
        \Sob{I}{L^p}&\le \left(\int_{|z|\le 2}\frac{1}{|z|^{\be}}\,dz\right)\Sob{\nabla g}{L^{q_1}}\Sob{f}{L^{r_1}}\notag\\
        &=\frac{C2^{{2-\be }}}{(2-\be)}\Sob{\nabla g}{L^{q_1}}\Sob{f}{L^{r_1}}.
    \end{align}
    We estimate $|II|$ by
    \begin{align*}
        |II|&\le \int_{|x-y|>1}\frac{1}{|x-y|^{1+\be}}|g(y)-g(x)||f(y)|\,dy\\
        &\le 2\Sob{g}{L^\infty}\int_{|x-y|>1}\frac{1}{|x-y|^{1+\be}}|f(y)|\,dy
    \end{align*}
    Applying Young's convolution inequality, we obtain
    \begin{align}\label{est:commutator:II}
        \Sob{II}{L^p}&\le 2\left(\int_{|z|> 1}\frac{1}{|z|^{p(1+\be)}}\,dz\right)^{\frac{1}{p}}\Sob{g}{L^\infty}\Sob{f}{L^{1}}\notag\\&=\frac{C}{((1+\be)p-2)^{\frac{1}{p}}}\Sob{g}{L^\infty}\Sob{f}{L^{1}}.
    \end{align}
    Now we estimate $|II|$ in a different way. Integrating by parts, we obtain
    \begin{align*}
       II&= \int_{\RR^2}\frac{(x-y)}{|x-y|^{2+\be}}(1-\rho(|x-y|))(g(y)-g(x))(-\Delta \Lam^{-2} f)(y)\,dy\\
        &=\int_{\RR^2}\bdy_l\left(\frac{(x-y)}{|x-y|^{2+\be}}\right)(1-\rho(|x-y|))(g(y)-g(x))(\bdy_l \Lam^{-2} f)(y)\,dy\\
        & \quad-\int_{\RR^2}\frac{(x-y)}{|x-y|^{2+\be}}(\bdy_l \rho(|x-y|))(g(y)-g(x))(\bdy_l \Lam^{-2} f)(y)\,dy\\
        &\quad+\int_{\RR^2}\frac{(x-y)}{|x-y|^{2+\be}}(1-\rho(|x-y|))(\bdy_l g(y))( \bdy_l \Lam^{-2} f)(y)\,dy\\
        &=II_a+II_b+II_c.
    \end{align*}
    We estimate $|II_a|$ by
    \begin{align*}
        |II_a|\le 2\Sob{g}{L^\infty}\int_{|x-y|>1}\frac{1}{|x-y|^{2+\be}}|(\Lam^{-1}\mathcal{R}_l f)(y)|\,dy.
    \end{align*}
    Applying Young's convolution inequality, we obtain
    \begin{align}\label{est:commutator:IIa}
        \Sob{II_a}{L^p}&\le 2\left(\int_{|z|> 1}\frac{1}{|z|^{2+\be}}\,dz\right)\Sob{g}{L^\infty}\Sob{\Lam^{-1}\mathcal{R}_l f}{L^{p}}\notag\\&\le\frac{C}{\be}\Sob{g}{L^\infty}\Sob{f}{\dot{W}^{-1,p}}.
    \end{align}
 Next, we estimate $|II_b|$. Since $|r\rho'(r)|\le 4$, we have
 \begin{align*}
     |II_b|\le 4\Sob{g}{L^\infty}\int_{2\ge |x-y|>1}\frac{1}{|x-y|^{2+\be}}|(\Lam^{-1}\mathcal{R}_l f)(y)|\,dy.
 \end{align*}
 Applying Young's convolution inequality and proceeding just like for $II_a$, we obtain
    \begin{align}\label{est:commutator:IIb}
        \Sob{II_b}{L^p}&\le\frac{C}{\be}\Sob{g}{L^\infty}\Sob{f}{\dot{W}^{-1,p}}.
    \end{align}
    Finally, we have
    \begin{align*}
        |II_c|\le \int_{ |x-y|>1}\frac{1}{|x-y|^{1+\be}}|\nabla g(y)||(\Lam^{-1}\mathcal{R}_l f)(y)|\,dy.
    \end{align*}
    Applying Young's convolution inequality and H\"older's inequality, we obtain
    \begin{align}\label{est:commutator:IIc}
        \Sob{II_c}{L^p}&\le \left(\int_{|z|> 1}\frac{1}{|z|^{\til{q_2}(1+\be)}}\,dz\right)^{\frac{1}{\til{q}_2}}\Sob{\nabla g}{L^{q_2}}\Sob{\Lam^{-1}\mathcal{R}_l f}{L^{r_2}}\notag\\&\le\frac{C}{((1+\be)\til{q}_2-2)^{\frac{1}{\til{q}_2}}}\Sob{\nabla g}{L^{q_2}}\Sob{f}{\dot{W}^{-1,r_2}},
    \end{align}
    where $\til{q}_2, q_2$, and $r_2$ satisfy
    \[1+\frac{1}{p}=\frac{1}{\til{q_2}}+\frac{1}{q_2}+\frac{1}{r_2},\quad \til{q}_2>1.\]
Collecting the estimates in \eqref{est:commutator:I}-\eqref{est:commutator:IIc}, we obtain the desired bound.
\end{proof}
\section{Conservation law with modified flux}
To prove our main results, we will recast \eqref{gSQG} as a particular case of a general equation that is of the form of a conservation law with modified flux. This resetting of \eqref{gSQG} illuminates the hidden cancellation structure which plays a crucial role in the control of nonlinear terms. This idea was first employed in \cite{JollyKumarMartinez2020a} in the form of a modified approximation scheme by the present author along with M.S. Jolly and V.R. Martinez to obtain local well-posedness for large data and global well-posedness for small data for dissipative gSQG equations in critical Sobolev spaces. Subsequently, it was also used effectively in \cite{JollyKumarMartinez2020b} to establish local well-posedness for logarithmically regularized versions of \eqref{gSQG}. As we will observe, this formulation also considerably simplifies statements of several key commutator estimates, thus providing a clear view of the main technical difficulties in the proofs. To be specific, given $q$ sufficiently smooth, we consider the following initial value problem:
\begin{align}\label{conservation:law}
\begin{cases}
     &\bdy_t \tht+\Div F_q^{\be} (\tht)=0,\\
     &\tht(x,0)=\tht_0(x).
\end{cases} 
\end{align}
where
\begin{align*}
    F_q^{\be}(\tht)=(\nabla^{\perp}\Lam^{\be-2}q)\tht+\Lam^{\be-2}((\nabla^\perp \tht)q).
\end{align*}
Observe that one formally has 
\begin{align*}
    \Div F_{-\tht}^{\be}(\tht)=-(\nabla^\perp \Lam^{\be-2}\tht)\cdot \nabla \tht=u\cdot \nabla \tht.
\end{align*}
Thus, one recovers equation \eqref{gSQG} in the case $q=-\tht$. Now, we will establish estimates for the nonlinear term in \eqref{conservation:law}. These are stated in \cref{lem:law:1} and \cref{lem:law:2}. The results obtained will then be employed to obtain the desired control on the nonlinear terms in the energy arguments required in the proofs of \cref{main:T:1} and \cref{main:T:2}. We remark that the commutator structure of the nonlinearity in \eqref{conservation:law} plays a central role in the proof of these estimates. The first estimate, stated in \cref{lem:law:1}, provides a sharper bound due to trilinear interactions of the nonlinearity as they appear in the energy arguments, while the second estimate \cref{lem:law:2} is a continuity estimate of the bilinear term. Furthermore, the bounds are valid in the limit $\be \to 1$.
\begin{lemma}\label{lem:law:1} Let $\be\in (1,2)$ and $\sigma>3$. Then we have 
\begin{align*}
    |\lb \Div F_q^\be (\tht), (-\Delta)^\sigma \tht\rb|\le C_6(\be,\si) \Sob{q}{H^\si}\Sob{\tht}{H^\si}^2,
\end{align*}
where $C_6$ is a positive constant. Furthermore, $C_6$ remains uniformly bounded as $\be \to 1$.
\end{lemma}
\begin{proof}
    We have
    \begin{align*}
        \lb \Div F_q^\be (\tht), (-\Delta)^\sigma \tht\rb&= \underbrace{\lb \Lam^\si (\nabla^{\perp}\Lam^{\be-2}q\cdot \nabla \tht),\Lam^\si \tht \rb}_{I^a}-\underbrace{\lb \Lam^{\si+\be-2}(\nabla^{\perp}q\cdot \nabla \tht), \Lam^\si \tht\rb}_{I^b}\\
        &=I_1+I_2+I_3+I_4+I_5,
    \end{align*}
where
\begin{align*}
    I_1&=\underbrace{\lb \nabla^{\perp}\Lam^{\be-2}\Lam^\si q\cdot \nabla \tht ,\Lam^\si \tht\rb}_{I_1^a}-\underbrace{\lb\nabla^{\perp}\Lam^{\be-2}\cdot(\Lam^\si q \nabla \tht) ,\Lam^\si \tht\rb}_{I^b_1},\\
    I_2&=\lb \nabla^{\perp}\Lam^{\be-2}q\cdot \nabla \Lam^{\si}\tht,\Lam^\si \tht \rb,\\
    I_3&=I^a-I_1^a-I_2\\
    &=\lb \Lam^\si (\nabla^{\perp}\Lam^{\be-2}q\cdot \nabla \tht),\Lam^\si \tht \rb-\lb \nabla^{\perp}\Lam^{\be-2}\Lam^\si q\cdot \nabla \tht ,\Lam^\si \tht\rb-\lb \nabla^{\perp}\Lam^{\be-2}q\cdot \nabla \Lam^{\si}\tht,\Lam^\si \tht \rb,\\
    I_4&=\lb \Lam^{\frac{\be}{2}-1}(\nabla^\perp q\cdot \nabla \Lam^{\frac{\be}{2}-1}\Lam^\si \tht ),\Lam^\si \tht\rb,\\
    I_5&=-I^b+I^b_1+I_4\\
    &=- \left (\lb \Lam^{\si+\be-2}(\nabla^{\perp}q\cdot \nabla \tht), \Lam^\si \tht\rb- \lb\nabla^{\perp}\Lam^{\be-2}\cdot(\Lam^\si q \nabla \tht) ,\Lam^\si \tht\rb- \lb \Lam^{\frac{\be}{2}-1}(\nabla^\perp q\cdot \nabla \Lam^{\frac{\be}{2}-1}\Lam^\si \tht ),\Lam^\si \tht\rb\right)
\end{align*}
Observe that by integrating by parts, we obtain $I_2,I_4=0$. 
\subsubsection{$\text{Estimating }I_1$} Let us denote by $\bdy^\perp=(-\bdy_{x_2},\bdy_{x_1})$. Then, we can rewrite $I_1$ as
\begin{align*}
    I_1=-\lb [\Lam^{\be-2}\bdy_\ell^\perp,\bdy_\ell \tht]\Lam^{\si}q,\Lam^\si \tht \rb.
\end{align*}
Applying \cref{commutator:estimate:1}, \eqref{commutator:1:est:2} with $f=\Lam^\si q$, $g=\bdy_\ell \tht$,  and $p=2, q_1=\infty, r_1=2, q_2=r_2=2$ , we obtain
\begin{align}
    |I_1|\le & C_2\Sob{ \nabla^2 \tht}{L^\infty}\Sob{q}{\Hdot^\si}\Sob{ \tht}{\Hdot^\si}+C_3 \Sob{\nabla \tht}{L^\infty}\Sob{q}{\dot{H}^{\si-1}}\Sob{ \tht}{\Hdot^{\si}}\notag\\
    & + C_4\Sob{ \nabla^2 \tht}{L^{2}}\Sob{q}{\Hdot^{\si-1}}\Sob{ \tht}{\Hdot^{\si}}\label{est:I1:first}\\
    \le & C\Sob{q}{H^\si}\Sob{\tht}{H^\si}^2 \label{est:I1:second}.
\end{align}
\subsubsection{$\text{Estimating }I_3$} Observe that for any $\til{\si}>2$, we have
\begin{align}\label{Laplacian:split}
    \Lam^{\til{\si}}f=\Lam^{\til{\si}-2}(-\Delta)f=-(\Lam^{\til{\si}-2}\bdy_l)\bdy_l f.
\end{align}
Then, by applying \eqref{Laplacian:split} and the product rule, we obtain
\begin{align*}
    I_3&=-\lb \Lam^{\si-2}\bdy_l(\bdy_l\bdy_\ell^\perp \Lam^{\be-2}q\bdy_\ell \tht),\Lam^\si \tht\rb+\lb (\Lam^{\si-2}\bdy_l(\bdy_l\bdy_\ell^\perp\Lam^{\be-2}q))\bdy_\ell \tht,\Lam^\si \tht \rb\\
    &-\lb \Lam^{\si-2}\bdy_l(\bdy_\ell^\perp \Lam^{\be-2}q\bdy_l \bdy_\ell \tht),\Lam^\si \tht\rb+\lb \bdy_\ell^{\perp}\Lam^{\be-2}q(\Lam^{\si-2}\bdy_l(\bdy_l\bdy_\ell \tht)), \Lam^\si \tht\rb\\
    &=-\lb [\Lam^{\si-2}\bdy_l, \bdy_\ell \tht]\bdy_l\bdy_\ell^\perp\Lam^{\be-2} q,\Lam^\si \tht\rb-\lb [\Lam^{\si-2}\bdy_l, \bdy_\ell^\perp \Lam^{\be-2}q]\bdy_l \bdy_\ell \tht,\Lam^\si \tht\rb\\
    &=I_3^a+I_3^b.
\end{align*}
Applying the Cauchy-Schwarz inequality followed by \eqref{est:KP:A} with $f=\bdy_\ell \tht$, $g=\bdy_l \bdy_\ell^\perp \Lam^{\be-2}q$, and $p=2$, we obtain
\begin{align}
    |I_3^a|&\le C(\Sob{\bdy_\ell \tht}{W^{\si-1,2}}\Sob{\bdy_l\bdy_\ell^\perp\Lam^{\be-2} q}{L^\infty}+\Sob{\bdy_\ell \tht}{W^{1,\infty}}\Sob{\bdy_l\bdy_\ell^\perp\Lam^{\be-2} q}{W^{\si-2,2}})\Sob{\Lam^\si \tht}{L^2}\notag\\
    &\le C(\Sob{\tht}{H^\si}\Sob{q}{H^{\be+1+\eps}}+\Sob{\tht}{H^{3+\eps}}\Sob{q}{H^{\si+\be-2}})\Sob{\Lam^\si \tht}{L^2}\label{est:I3a:first}\\
    &\le C\Sob{q}{H^{\si}}\Sob{\tht}{H^\si}^2\label{est:I3a:second}.
\end{align}
Applying the Cauchy-Schwarz inequality followed by \eqref{est:KP:A} with $f=\bdy_\ell^\perp \Lam^{\be-2}q$, $g=\bdy_l \bdy_\ell \tht$, and $p=2$, we obtain
\begin{align}
    |I_3^b|&\le C(\Sob{\bdy_\ell^\perp \Lam^{\be-2}q}{W^{\si-1,2}}\Sob{\bdy_l \bdy_\ell \tht}{L^\infty}+\Sob{\bdy_\ell^\perp \Lam^{\be-2}q}{W^{1,\infty}}\Sob{\bdy_l \bdy_\ell \tht}{W^{\si-2,2}})\Sob{\Lam^\si \tht}{L^2}\notag\\
    &\le C(\Sob{q}{H^{\si+\be-2}}\Sob{\tht}{^{H^{3+\eps}}}+\Sob{q}{H^{\be+1+\eps}}\Sob{\tht}{H^\si})\Sob{\Lam^\si \tht}{L^2}\label{est:I3b:first}\\
    &\le C\Sob{q}{H^\si}\Sob{\tht}{H^\si}^2\label{est:I3b:second}.
\end{align}
\subsubsection{$\text{Estimating }I_5$} Applying \eqref{Laplacian:split} and the product rule, we obtain
\begin{align*}
    I_5=&\lb \Lam^{\si-2}\bdy_l(\nabla^\perp(\bdy_l q)\cdot\nabla\tht),\Lam^{\si+\be-2}\tht \rb-\lb (\Lam^{\si-2}\nabla^\perp \Delta q)\cdot\nabla \tht, \Lam^{\si+\be-2}\tht \rb\\
    & +\lb \Lam^{\si +\frac{\be}{2}-3}\bdy_l(\nabla^\perp q \cdot \nabla (\bdy_l \tht)),\Lam^{\si+\frac{\be}{2}-1}\tht\rb-\lb \nabla^\perp q\cdot \nabla \Lam^{\si+\frac{\be}{2}-3}\Delta \tht, \Lam^{\si+\frac{\be}{2}-1}\tht\rb\\
    =&\lb [\Lam^{\si-2}\bdy_l,\bdy_\ell \tht]\bdy_\ell^\perp \bdy_l q, \Lam^{\si+\be-2}\tht\rb+\lb [\Lam^{\si+\frac{\be}{2}-3}\bdy_l,\bdy_\ell^\perp q]\bdy_\ell \bdy_l \tht,\Lam^{\si+\frac{\be}{2}-1}\tht \rb.
\end{align*}
Applying the Cauchy-Schwarz inequality followed by \eqref{est:KP:A} with $f=\bdy_\ell\tht$, $g= \bdy_\ell^\perp \bdy_l q$, and $p=2$, we obtain
\begin{align}
    |I_5^a|&\le C(\Sob{\bdy_\ell \tht}{W^{\si-1,2}}\Sob{\bdy_\ell^\perp \bdy_l q}{L^\infty}+\Sob{\bdy_\ell \tht}{W^{1,\infty}}\Sob{\bdy_\ell^\perp \bdy_l q}{W^{\si-2,2}})\Sob{\Lam^{\si+\be-2}\tht}{L^2}\notag\\
    &\le C(\Sob{\tht}{H^\si}\Sob{q}{H^{3+\eps}}+\Sob{\tht}{H^{3+\eps}}\Sob{q}{H^\si})\Sob{\Lam^{\si+\be-2}\tht}{L^2}\label{est:I5a:first}\\
    &\le C\Sob{q}{H^\si}\Sob{\tht}{H^\si}^2\label{est:I5a:second}.
\end{align}
Similarly, applying the Cauchy-Schwarz inequality followed by \eqref{est:KP:A} with $f=\bdy_\ell^\perp q$, $g=\bdy_l \bdy_\ell \tht$, and $p=2$, we obtain
\begin{align}
    |I^b_5|&\le C(\Sob{\bdy_\ell^\perp q}{W^{\si+\frac{\be}{2}-2,2}}\Sob{\bdy_\ell \bdy_l \tht}{L^\infty}+\Sob{\bdy_\ell^\perp q}{W^{1,\infty}}\Sob{\bdy_\ell \bdy_l \tht}{W^{\si+\frac{\be}{2}-3,2}})\Sob{\Lam^{\si+\frac{\be}{2}-1}\tht}{L^2}\notag\\
    &\le C\Sob{q}{H^{\si+\frac{\be}{2}-1}}\Sob{\tht}{H^{3+\eps}}+\Sob{q}{H^{3+\eps}}\Sob{\tht}{H^{\si+\frac{\be}{2}-1}}\Sob{\Lam^{\si+\frac{\be}{2}-1}\tht}{L^2}\label{est:I5b:first}\\
    &\le C\Sob{q}{H^\si}\Sob{\tht}{H^\si}^2\label{est:I5b:second}.
\end{align}
Collecting the estimates in \eqref{est:I1:second}, \eqref{est:I3a:second}, \eqref{est:I3b:second}, \eqref{est:I5a:second}, and \eqref{est:I5b:second}, we obtain the desired result.
\end{proof}
\begin{remark}
We remark that for $\si>4$ and $q=\tht$, it follows from \eqref{est:I1:first}, \eqref{est:I3a:first}, \eqref{est:I3b:first}, \eqref{est:I5a:first}, and \eqref{est:I5b:first}, in the proof of \cref{lem:law:1}, that
\begin{align}\label{est:sigma>4}
    |\lb \Div F_{\tht}^\be (\tht), (-\Delta)^\sigma \tht\rb|\le C_6(\be,\si) \Sob{\tht}{H^{\si-1}}\Sob{\tht}{H^\si}^2.
\end{align}
\end{remark}
We now state the continuity estimate alluded to earlier. With a constant which may become unbounded as $\be \to 1$, we obtain a bound only in terms of $\Sob{q}{H^\si}$ while additional $L^1$ regularity ensures a bound valid even when $\be \to 1$.
\begin{lemma}\label{lem:law:2} Let $\be\in (1,2)$ and $\sigma>3$. Then we have 
\begin{align*}
    \Sob{\Div F^\be_q(\tht)}{H^\si}\le \min \{C_7(\be,\si)\Sob{q}{H^\si},C_8(\be,\si)(\Sob{q}{H^\si}+\Sob{q}{L^1})\}\Sob{\tht}{H^{\si+1}},
\end{align*}
where $C_7, C_8$ are positive constants. Furthermore, $C_8$ remains uniformly bounded as $\be \to 1$.
\end{lemma}
\begin{proof}
    Let us denote by 
    \[\Gamma=\Div F^\be_q(\tht).\]
    We will estimate the norms $\Sob{\Gamma}{L^2}^2$ and $\Sob{\Lam^\si \Gamma}{L^2}^2$ separately. Observe that
    \[ \Sob{\Gamma}{L^2}^2=-\lb [\bdy_\ell^\perp \Lam^{\be-2},\bdy_\ell \tht]q,\Gamma\rb.\]
    Applying the Cauchy-Schwarz inequality followed by \cref{lem:Li2019} with $f=q$, $g=\bdy_\ell \tht$, and $p=2$, we obtain
    \begin{align}
        \Sob{\Gamma}{L^2}^2
        &\le C\Sob{\nabla^\perp \Lam^{\be-2}\nabla \tht}{L^\infty}\Sob{q}{L^2}\Sob{\Gamma}{L^2}\notag\\
        &\le C\Sob{\tht}{H^{\be+1+\eps}}\Sob{q}{L^2}\Sob{\Gamma}{L^2}\label{est:L2:first},
    \end{align}
    and applying the Cauchy-Schwarz inequality followed by \cref{commutator:estimate:1}, \eqref{commutator:1:est:1} with $f=q$, $g=\bdy_\ell \tht$, and $p=2, q_1=\infty, r_1=2$, we obtain
    \begin{align}
        \Sob{\Gamma}{L^2}^2
        \le & C_2\Sob{\nabla^2 \tht}{L^\infty}\Sob{q}{L^2}\Sob{\Gamma}{L^2}\notag\\
        & +C_3\Sob{\nabla \tht}{L^\infty}\Sob{q}{L^1}\Sob{\Gamma}{L^2}\notag\\
        \le & C\Sob{\tht}{H^\si}(\Sob{q}{L^2}+\Sob{q}{L^1})\Sob{\Gamma}{L^2}\label{est:L2:second}.
    \end{align}
    Next, we decompose $\Sob{\Lam^\si \Gamma}{L^2}^2$ as follows:
    \begin{align*}
        \Sob{\Lam^\si \Gamma}{L^2}^2&=\lb \Lam^\si (\nabla^\perp \Lam^{\be-2}q\cdot \nabla \tht),\Lam^\si \Gamma\rb-\lb \Lam^{\si+\be-2}(\nabla^\perp q \cdot \nabla \tht),\Lam^\si \Gamma\rb\\
        &=-\lb [\bdy_\ell^\perp\Lam^{\be-2},\bdy_\ell \tht]\Lam^\si q, \Lam^\si \Gamma\rb+ \lb [\Lam^\si, \bdy_\ell \tht]\bdy_\ell^\perp \Lam^{\be-2}q,\Lam^\si \Gamma\rb\\
        &\quad -\lb [\Lam^\si, \bdy_\ell \tht]\bdy_\ell^\perp q,\Lam^{\si+\be-2}\Gamma \rb\\
        &=J_1+J_2+J_3.
    \end{align*}
    Applying the Cauchy-Schwarz inequality followed by \cref{commutator:estimate:1}, \eqref{commutator:1:est:2} with $f=\bdy_\ell \tht$, $g=\Lam^\si q$, and $p=2, q_1=\infty, r_1=2, ,q_2=r_2=2$ , we obtain 
    \begin{align}\label{est:J1}
    |J_1|&\le C_2\Sob{ \nabla^2 \tht}{L^\infty}\Sob{q}{\Hdot^\si}\Sob{\Gamma}{\Hdot^\si}+C_3\Sob{\nabla \tht}{L^\infty}\Sob{q}{\Hdot^{\si-1}}\Sob{\Gamma}{\Hdot^\si}\notag\\
    &\quad + C_4\Sob{\nabla^2 \tht}{L^{2}}\Sob{q}{\dot{H}^{\si-1}}\Sob{ \Gamma}{\Hdot^{\si}}\notag\\
    &\le C\Sob{q}{H^\si}\Sob{\tht}{H^\si}\Sob{\Gamma}{H^\si}.
\end{align}
Applying the Cauchy-Schwarz inequality followed by \eqref{est:KP:A} with $f=\bdy_\ell \tht$, $g=\bdy_\ell^\perp \Lam^{\be-2}q$, and $p=2$, we obtain
    \begin{align}\label{est:J2}
        |J_2|&\le C(\Sob{\bdy_\ell \tht}{W^{\si,2}}\Sob{\bdy_\ell^\perp q}{L^\infty}+\Sob{\bdy_\ell \tht}{W^{1,\infty}}\Sob{\bdy_\ell^\perp q}{W^{\si-1,2}})\Sob{\Lam^\si \Gamma}{L^2}\notag\\
        &\le C(\Sob{\tht}{H^{\si+1}}\Sob{q}{H^{2+\eps}}+\Sob{\tht}{H^{3+\eps}}\Sob{q}{H^\si})\Sob{\Lam^\si \Gamma}{L^2}\notag\\
        &\le C\Sob{q}{H^\si}\Sob{\tht}{H^{\si+1}}\Sob{\Lam^\si \Gamma}{L^2}.
    \end{align}
    Similarly, applying the Cauchy-Schwarz inequality followed by \eqref{est:KP:A} with $f=\bdy_\ell \tht$, $g=\bdy_\ell^\perp q$, and $p=2$, we obtain
    \begin{align}\label{est:J3}
        |J_3|&\le C(\Sob{\bdy_\ell \tht}{W^{\si,2}}\Sob{\bdy_\ell^\perp q}{L^\infty}+\Sob{\bdy_\ell \tht}{W^{1,\infty}}\Sob{\bdy_\ell^\perp q}{W^{\si-1,2}})\Sob{\Lam^{\si+\be-2}\tht}{L^2}\notag\\
        &\le C(\Sob{\tht}{H^{\si+1}}\Sob{q}{H^{2+\eps}}+\Sob{\tht}{H^{3+\eps}}\Sob{q}{H^\si})\Sob{\Lam^{\si+\be-2}\tht}{L^2}\notag\\
        &\le C\Sob{q}{H^\si}\Sob{\tht}{H^{\si+1}}\Sob{\Gamma}{H^\si}.
    \end{align}
    From the estimates in \eqref{est:L2:first}, \eqref{est:J1}-\eqref{est:J3}, we obtain
    \begin{align*}
        \Sob{\Gamma}{H^\si}^2 \le C_5C\Sob{q}{H^\si}\Sob{\tht}{H^{\si+1}}\Sob{\Gamma}{H^\si},
    \end{align*}
    and from the estimates in \eqref{est:L2:second}, \eqref{est:J1}-\eqref{est:J3}, we obtain
    \begin{align*}
        \Sob{\Gamma}{H^\si}^2\le C_6(\Sob{q}{H^\si}+\Sob{q}{L^1})\Sob{\tht}{H^{\si+1}}\Sob{\Gamma}{H^\si},
    \end{align*}
    where $C_6$ is uniformly bounded as $\be\to 1$. This completes the proof.
\end{proof}
\section{Proof of \cref{main:T:1}}
We will now carry out the proof of \cref{main:T:1}. We do so by studying the evolution of the difference of $\tht^\be$ and $\tht^{\be_0}$.
Denote by
\[\overline{\tht}=\tht^{\be}-\tht^{\be_0}.\]
Then, $\overline{\tht}$ satisfies the equation
\begin{align*}
    \bdy_t \overline{\tht}+u^\be\cdot \nabla \tht^\be-u^{\be_0} \cdot \nabla \tht^{\be_0}.
\end{align*}
Observe that 
\begin{align}\label{theta:bar:equation}
    u^\be\cdot \nabla \tht^\be-u^{\be_0} \cdot \nabla \tht^{\be_0}&=\Div F^\be_{-\tht^\be}(\tht^\be)-\Div F^{\be_0}_{-\tht^{\be_0}}(\tht^{\be_0})\notag\\
    &=\Div F^\be_{-\overline{\tht}}(\overline{\tht})+\Div F^\be_{-\tht^{\be_0}}(\overline{\tht})+\Div F^\be_{-\overline{\tht}}(\tht^{\be_0})\notag\\
    &\quad+(\nabla^\perp \Lam^{\be-2}-\nabla^{\perp}\Lam^{\be_0 -2})\tht^{\be_0}\cdot\nabla \tht^{\be_0}.
\end{align}
Let us denote by 
\[\mathcal{H}^{\be,\be_0}(\tht^{\be_0})=(\nabla^\perp \Lam^{\be-2}-\nabla^{\perp}\Lam^{\be_0 -2})\tht^{\be_0}.\]
Taking the $L^2-$ inner product of \eqref{theta:bar:equation} with $\overline{\tht}$, we obtain
\begin{align}
    \frac{1}{2}\frac{d}{dt}\Sob{\overline{\tht}(t)}{L^2}^2=&\lb \Div F^\be_{\overline{\tht}}(\overline{\tht}),\overline{\tht} \rb+ \lb \Div F^\be_{\tht^{\be_0}}(\overline{\tht}),\overline{\tht} \rb\notag\\
    &+\lb \Div F^\be_{\overline{\tht}}(\tht^{\be_0}), \overline{\tht} \rb -\lb \mathcal{H}^{\be,\be_0}(\tht^{\be_0})\cdot\nabla \tht^{\be_0}, \overline{\tht}\rb\notag\\
    =&K_1+K_2+K_3+K_4. \label{energy:K}
\end{align}
Applying the Cauchy-Schwarz inequality followed by \cref{lem:Li2019} with $f=\overline{\tht}$, $g=\nabla \overline{\tht}$, and $p=2$, we obtain
\begin{align}|K_1|=|\lb [\nabla^\perp \Lam^{\be-2}\cdot, \nabla \overline{\tht}]\overline{\tht},\overline{\tht} \rb|&\le C\Sob{\Lam^{\be-1}\nabla \overline{\tht}}{L^\infty}\Sob{\bar{
\tht
}}{L^2}^2\notag\\
&\le  C\Sob{\overline{\tht}}{H^{\be+1+\eps}}\Sob{\bar{
\tht
}}{L^2}^2\notag\\&
\le  C\Sob{\overline{\tht}}{H^{s}}\Sob{\bar{
\tht
}}{L^2}^2.\label{est:K1}
\end{align}
Similarly, we obtain
\begin{align}
    |K_2|\le & C\Sob{\overline{\tht}}{H^{\be+1+\eps}}\Sob{\tht^{\be_0}}{L^2}\Sob{\overline{\tht}}{L^2}\notag\\
    \le &C\Sob{\overline{\tht}}{H^{s}}\Sob{\tht^{\be_0}}{L^2}\Sob{\overline{\tht}}{L^2},\label{est:K2}
\end{align}
and
\begin{align}
    |K_3|\le & C\Sob{\tht^{\be_0}}{H^{\be+1+\eps}}\Sob{\overline{\tht}}{L^2}^2\notag\\
    \le & C\Sob{\tht^{\be_0}}{H^{s}}\Sob{\overline{\tht}}{L^2}^2.\label{est:K3}
\end{align}
By the Cauchy-Schwarz inequality, we obtain
\begin{align*}
    |K_4|\le \Sob{\mathcal{H}^{\be,\be_0}(\tht^{\be_0})\cdot\nabla \tht^{\be_0}}{L^2}\Sob{\overline{\tht}}{L^2}.
\end{align*}
Next, taking the inner product of \eqref{theta:bar:equation} with $(-\Delta)^s \overline{\tht}$, we obtain
\begin{align}
    \frac{1}{2}\frac{d}{dt}\Sob{\overline{\tht}(t)}{\Hdot^s}^2=&\lb \Div F^\be_{\overline{\tht}}(\overline{\tht}),(-\Delta)^s \overline{\tht} \rb+ \lb \Div F^\be_{\tht^{\be_0}}(\overline{\tht}),(-\Delta)^s \overline{\tht} \rb\notag\\
    &+\lb \Div F^\be_{\overline{\tht}}(\tht^{\be_0}),(-\Delta)^s \overline{\tht} \rb -\lb \mathcal{H}^{\be,\be_0}(\tht^{\be_0})\cdot\nabla \tht^{\be_0},(-\Delta)^s \overline{\tht}\rb\notag\\
    = &L_1+L_2+L_3+L_4.\label{energy:L}
 \end{align}
 Applying \cref{lem:law:1} with $q=\tht=\overline{\tht}$, and $\si=s$, we obtain
 \begin{align}\label{est:L1}
     |L_1|& \le C\Sob{\overline{\tht}}{H^s}^3.
     \end{align}
Similarly, applying \cref{lem:law:1} with $q=\tht^{\be_0}$, $\tht=\overline{\tht}$, and $\si=s$, we obtain
     \begin{align}\label{est:L2}
     |L_2|&\le C\Sob{\tht^{\be_0}}{H^s}\Sob{\overline{\tht}}{H^s}^2.
 \end{align}
 Applying the Cauchy-Schwarz inequality followed by \cref{lem:law:2} with $q=\overline{\tht}$, $\tht=\tht^{\be_0}$, and $\si=s$, we obtain
 \begin{align}\label{est:L3}
     |L_3|\le C\Sob{\tht^{\be_0}}{H^{s+1}}\Sob{\overline{\tht}}{H^s}^2.
 \end{align}
 Applying the Cauchy-Schwarz inequality, we obtain
 \begin{align*}
     |L_4|\le C\Sob{\mathcal{H}^{\be,\be_0}(\tht^{\be_0})\cdot\nabla \tht^{\be_0}}{\Hdot^s}\Sob{\overline{\tht}}{\Hdot^s}.
 \end{align*}
 Applying \eqref{est:Leibnitz} with $p_1=p_3=2$ and $p_2=p_4=\infty$, we obtain
 \begin{align}
     |K_4|, |L_4|&\le C\Sob{\mathcal{H}^{\be,\be_0}(\tht^{\be_0})\cdot\nabla \tht^{\be_0}}{H^s}\Sob{\overline{\tht}}{H^s}\notag\\
     &\le C(\Sob{\mathcal{H}^{\be,\be_0}(\tht^{\be_0})}{H^s}\Sob{\nabla \tht^{\be_0}}{L^\infty}+\Sob{\mathcal{H}^{\be,\be_0}(\tht^{\be_0})}{L^\infty}\Sob{\nabla \tht^{\be_0}}{H^s})\Sob{\overline{\tht}}{H^s}\notag\\
     &\le C\Sob{\mathcal{H}^{\be,\be_0}(J^s\tht^{\be_0})}{L^2}\Sob{\tht^{\be_0}}{H^{s+1}}\Sob{\overline{\tht}}{H^s},\label{est:KL4}
 \end{align}
 where we used the fact that $H^s \hookrightarrow L^\infty$ for $s>1$.
 Plugging the estimates \eqref{est:K1}-\eqref{est:K3}, and \eqref{est:KL4} in \eqref{energy:K}, we obtain
 \begin{align}\label{est:energy:K}
     \frac{1}{2}\frac{d}{dt}\Sob{\overline{\tht}(t)}{L^2}^2\le& C\Sob{\overline{\tht}}{H^{s}}\Sob{\bar{
\tht
}}{L^2}^2+C\Sob{\overline{\tht}}{H^{s}}\Sob{\tht^{\be_0}}{L^2}\Sob{\overline{\tht}}{L^2}+C\Sob{\tht^{\be_0}}{H^{s}}\Sob{\overline{\tht}}{L^2}^2\notag\\&+C\Sob{\mathcal{H}^{\be,\be_0}(J^s\tht^{\be_0})}{L^2}\Sob{\tht^{\be_0}}{H^{s+1}}\Sob{\overline{\tht}}{H^s}.
 \end{align}
 Similarly, plugging the estimates \eqref{est:L1}-\eqref{est:L3}, and \eqref{est:KL4} in \eqref{energy:L}, we obtain
 \begin{align}  \label{est:energy:L}
 \frac{1}{2}\frac{d}{dt}\Sob{\overline{\tht}(t)}{\Hdot^s}^2 \le & C\Sob{\overline{\tht}}{H^s}^3+C\Sob{\tht^{\be_0}}{H^s}\Sob{\overline{\tht}}{H^s}^2+C\Sob{\tht^{\be_0}}{H^{s+1}}\Sob{\overline{\tht}}{H^s}^2\notag\\
 &+C\Sob{\mathcal{H}^{\be,\be_0}(J^s\tht^{\be_0})}{L^2}\Sob{\tht^{\be_0}}{H^{s+1}}\Sob{\overline{\tht}}{H^s}.
 \end{align}
 Adding \eqref{est:energy:K} and \eqref{est:energy:L} together, we obtain
 \begin{align}\label{est:energy:K+L}
     \frac{d}{dt}\Sob{\overline{\tht}(t)}{H^s}\le C\Sob{\overline{\tht}}{H^s}^2+C\Sob{\tht^{\be_0}}{H^{s+1}}\Sob{\overline{\tht}}{H^s}+C\Sob{\mathcal{H}^{\be,\be_0}(J^s\tht^{\be_0})}{L^2}\Sob{\tht^{\be_0}}{H^{s+1}}.
 \end{align}
 Now, we will obtain an estimate for $\Sob{\mathcal{H}^{\be,\be_0}(J^s\tht^{\be_0})}{L^2}$. First, we recall the singular integral representation of $\mathcal{H}^{\be,\be_0}(\tht^{\be_0})$: 
 \begin{align*}
     \mathcal{H}^{\be,\be_0}(\tht^{\be_0})(x)=\text{P.V. }C_\be \int_{\RR^2}\left(\frac{(x-y)^{\perp}}{|x-y|^{2+\be}}-\frac{(x-y)^{\perp}}{|x-y|^{2+\be_0}}\right)\tht^{\be_0}(y)\,dy.
 \end{align*}
 For simplicity, we will take $C_\be=1$. For $0<\veps<1$, to be determined later, we write
 \begin{align}
     \mathcal{H}^{\be,\be_0}(J^s\tht^{\be_0})(x)&=\int_{\RR^2}\left(\frac{(x-y)^{\perp}}{|x-y|^{2+\be}}-\frac{(x-y)^{\perp}}{|x-y|^{2+\be_0}}\right)J^s\tht^{\be_0}(y)\,dy\notag\\
     &=\left(\int_{|x-y|\le \veps}+\int_{1>|x-y|\ge \veps}+\int_{|x-y|\ge 1} \right)\notag\\
     &\quad \left(\frac{(x-y)^{\perp}}{|x-y|^{2+\be}}-\frac{(x-y)^{\perp}}{|x-y|^{2+\be_0}}\right)J^s \tht^{\be_0}(y)\,dy]\notag\\
     &=M_1+M_2+M_3.\label{energy:H:M}
 \end{align}
 Using the fact that 
 \begin{align*}
     \int_{|z|=1}\frac{z^\perp}{|z|^{2+\be}}\,ds=\int_{|z|=1}\frac{z^\perp}{|z|^{2+\be_0}}\,ds=0,
 \end{align*}
 we can express $M_1$ and $M_2$ as 
 \begin{align*}
     M_1=\int_{|x-y|\le \veps}\left(\frac{(x-y)^{\perp}}{|x-y|^{2+\be}}-\frac{(x-y)^{\perp}}{|x-y|^{2+\be_0}}\right)\left(J^s \tht^{\be_0}(y)-J^s \tht^{\be_0}(x)\right)\,dy,
 \end{align*}
 and 
 \begin{align*}
     M_2=\int_{1>|x-y|\ge \veps}\left(\frac{(x-y)^{\perp}}{|x-y|^{2+\be}}-\frac{(x-y)^{\perp}}{|x-y|^{2+\be_0}}\right)\left(J^s \tht^{\be_0}(y)-J^s \tht^{\be_0}(x)\right)\,dy,
 \end{align*}
 Using the mean value theorem and making a change of variables, we obtain
 \begin{align*}
     |M_1|\le \int_0^1\int_{|z|\le \veps}\left(\frac{1}{|z|^{\be}}+\frac{1}{|z|^{\be_0}}\right)|\nabla J^s\tht^{\be_0}(x-\tau z)|\,dz\,d\tau.
 \end{align*}
 Taking the $L^2-$norm on both sides and applying the Minkowski inequality, we obtain
 \begin{align*}
     \Sob{M_1}{L^2}\le \int_0^1\int_{|z|\le \veps}\left(\frac{1}{|z|^{\be}}+\frac{1}{|z|^{\be_0}}\right)\Sob{\nabla J^s \tht^{\be_0}(\cdot-\tau z)}{L^2}\,dz\,d\tau.
 \end{align*}
 Using the translation-invariance of the Lebesgue measure, we conclude that
 \begin{align}\label{est:M1}
     \Sob{M_1}{L^2}\le C\left(\frac{\veps^{2-\be}}{2-\be}+\frac{\veps^{2-\be_0}}{2-\be_0}\right)\Sob{J^{s+1}\tht^{\be_0}}{L^2}.
 \end{align}
 To estimate $M_2$ and $M_3$, we consider the two cases:
 \subsubsection{Case $\be_0>\be$:} Using the mean value theorem, we obtain
 \begin{align*}
     |M_2|\le|\be-\be_0| \int_{1>|x-y|\ge \veps}\frac{|\ln|x-y||}{|x-y|^{1+\be_0}}|J^s \tht^{\be_0}(y)- J^s \tht^{\be_0}(x)|\,dy.
 \end{align*}
 Using the mean value theorem, making a change of variables, and applying the Minkowski inequality, we obtain
 \begin{align}
     \Sob{M_2}{L^2}&\le |\be-\be_0|\left(\int_{1>|z|\ge \veps}\frac{|\ln|z||}{|z|^{\be_0}}\,dz\right)\Sob{J^{s+1}\tht^{\be_0}}{L^2}\notag\\
     &\le \frac{C}{2-\be_0}|\be-\be_0||\ln \veps|\Sob{\tht^{\be_0}}{H^{s+1}}\label{est:M2:case1}.
 \end{align}
 Let $\kap\in (0,(\be-1)/2)$. Then, we have $\ln|x-y|\le C|x-y|^\kap$ for $|x-y|\ge 1$. Therefore, we have
 \begin{align*}
     |M_3|\le |\be-\be_0|\int_{|x-y|\ge 1}\frac{|\ln|x-y||}{|x-y|^{1+\be}}|J^s \tht^{\be_0}(y)|\,dy.
 \end{align*}
 Making a change of variables and applying the Young's convolution inequality, we obtain
 \begin{align}\label{est:M3:case1}
     \Sob{M_3}{L^2}&\le |\be-\be_0|\left(\int_{|z|\ge 1}\frac{|\ln|z||}{|z|^{1+\be}}\,dz\right)\Sob{J^s\tht^{\be_0}}{L^2}\notag\\
     &\le |\be-\be_0|\left(\int_{|z|\ge 1}\frac{1}{|z|^{1+\be-\kap}}\,dz\right)\Sob{J^s \tht^{\be_0}}{L^2}\notag\\
     &\le \frac{C}{\be-1}|\be-\be_0|\Sob{\tht^{\be_0}}{H^s}.
 \end{align}
 \subsubsection{Case $\be>\be_0$:} This case is treated in a similar way by exchanging the positions of $\be$ and $\be_0$. In this case, we obtain the following bounds:
 \begin{align}
      \Sob{M_2}{L^2}&\le  \frac{C}{2-\be}|\be-\be_0||\ln \veps|\Sob{\tht^{\be_0}}{H^{s}},\label{est:M2:case2}\\
      \Sob{M_3}{L^2}&\le \frac{C}{\be_0-1}|\be-\be_0|\Sob{\tht^{\be_0}}{H^s}.\label{est:M3:case2}
 \end{align}
 Collecting the estimates in \eqref{est:M1}-\eqref{est:M3:case2}, we obtain
 \begin{align}\label{est:H}
     \Sob{\mathcal{H}^{\be,\be_0}(J^s\tht^{\be_0})}{L^2}\le C(\veps^{2-\be}+\veps^{2-\be_0}+|\be-\be_0|+|\be-\be_0||\ln \veps|)\Sob{\tht^{\be_0}}{H^{s+1}}.
 \end{align}
 
 We now set $\veps=|\be-\be_0|$ and use \eqref{est:H} in \eqref{est:energy:K+L}, to obtain
 \begin{align}\label{est:energy:K+L:2}
     \frac{d}{dt}\Sob{\overline{\tht}(t)}{H^s}\le& C\Sob{\overline{\tht}}{H^s}^2+C\Sob{\tht^{\be_0}}{H^{s+1}}\Sob{\overline{\tht}}{H^s}\notag\\
     &+C\left(|\be-\be_0|^{2-\be}+|\be-\be_0|^{2-\be_0}+|\be-\be_0||\ln |\be-\be_0||\right)\Sob{\tht^{\be_0}}{H^{s+1}}^2.
 \end{align}
 Denote by
 \begin{align*}
     y(t)=\Sob{\overline{\tht}}{H^s}\text{exp}\left(-C\int_0^t\Sob{\tht^{\be_0}}{H^{s+1}}\,ds\right).
 \end{align*}
Multiplying \eqref{est:energy:K+L:2} by $\text{exp}\left(-C\int_0^t\Sob{\tht^{\be_0}}{H^{s+1}}\,ds\right)$, we obtain
 \begin{align*}
     \frac{dy(t)}{dt}\le C\left(|\be-\be_0|^{2-\be}+|\be-\be_0|^{2-\be_0}+|\be-\be_0||\ln |\be-\be_0||\right)A(t)+By^2(t),
 \end{align*}
 where 
 \begin{align*}
     A(t)=\Sob{\tht^{\be_0}}{H^{s+1}}^2\text{exp}\left(-C\int_0^t\Sob{\tht^{\be_0}}{H^{s+1}}\,ds\right),
 \end{align*}
 and
 \begin{align*}
     B=C\text{exp}\left(C\int_0^T\Sob{\tht^{\be_0}}{H^{s+1}}\,ds\right).
 \end{align*}
 Applying \cref{lem:Gronwall:variant} with $F(t)=A(t)$ and $G=B$, we deduce that there exists a $\delta>0$ depending on $T$ and $\int_0^T \Sob{\tht^{\be_0}}{H^{s+1}}\,dt$ such that for $\be \in (1,2)$ and satisfying $|\be-\be_0|<\delta$, we have
 \begin{align*}
     y(t)\le C\left(|\be-\be_0|^{2-\be}+|\be-\be_0|^{2-\be_0}+|\be-\be_0||\ln |\be-\be_0||\right)\int_0^T A(t)\,dt.
 \end{align*}
 Therefore, we have
 \begin{align*}
     \Sob{\overline{\tht}}{H^s}\le C\left(|\be-\be_0|^{2-\be}+|\be-\be_0|^{2-\be_0}+|\be-\be_0||\ln |\be-\be_0||\right),
 \end{align*}
 where $C$ is a constant depending on $T$ and $\int_0^T \Sob{\tht^{\be_0}}{H^{s+1}}\,dt$. \par
 Let $\tht^{\be_0}\in C([0,T_0];H^{s+1})$, $s>3$. We will now show that the interval of existence of $\tht^\be$ for $\be$ close to $\be_0$ contains $[0,T_0]$. We proceed as in \cite{YuZhengJiu2019}. For $\be \in (1,2)$ such that $|\be-\be_0|<\delta$, let $\tht^\be \in C([0,T^\be_{\text{max}});H^{s+1})$, where $T^\be_{\text{max}}$ is the maximal time of existence. If $T^\be_{\text{max}}\ge T_0$, then the claim follows. So, let us assume that $T^\be_{\text{max}}<T_0$. Taking the $L^2-$ inner product of \eqref{gSQG} with $\tht^\be$, and applying the Cauchy-Schwarz inequality followed by \cref{lem:Li2019} with $g=\nabla \tht^\be, f=\tht^\be$, and $p=2$, we obtain
 \begin{align}\label{est:energy:s+1:first}
       \frac{d}{dt}\Sob{{\tht^\be}(t)}{L^2}^2&\le |\lb [\nabla^\perp \Lam^{\be-2}\cdot,\nabla \tht^\be] \tht^\be,\tht^\be\rb|\notag\\
       &\le C\Sob{\nabla \Lam^{\be-1}\tht^\be}{L^\infty}\Sob{\tht^\be}{L^2}^2\notag\\
       &\le C\Sob{\tht^\be}{H^s}\Sob{\tht^\be}{L^2}^2.
 \end{align}
 Next, taking the $L^2-$ inner product of \eqref{gSQG} with $(-\Delta)^{s+1}\tht^\be$ and using \eqref{est:sigma>4} with $\si=s+1$, we obtain
 \begin{align}\label{est:energy:s+1:second}
      \frac{d}{dt}\Sob{{\tht^\be}(t)}{\Hdot^{s+1}}^2&\le |\lb \Div F^\be_{\tht^\be}(\tht^\be), (-\Delta)^{s+1}\tht^\be\rb|\notag\\
      &\le C\Sob{\tht^\be}{H^s}\Sob{\tht^\be}{H^{s+1}}^2.
 \end{align}
 From \eqref{est:energy:s+1:first} and \eqref{est:energy:s+1:second}, we have
 \begin{align*}
     \frac{d}{dt}\Sob{{\tht^\be}(t)}{H^{s+1}}^2 \le C\Sob{\tht^\be}{H^s}\Sob{\tht^\be}{H^{s+1}}^2.
 \end{align*}
By Gronwall inequality, we obtain
\begin{align*}
    \Sob{\tht^\be(t)}{H^{s+1}}&\le C\text{exp}\left( \int_0^{T^\be_{\text{max}}}\Sob{\tht^\be(t)}{H^{s}}\,dt\right)\Sob{\tht_0}{H^{s+1}}\\
    &\le C\text{exp}\left( \int_0^{T^\be_{\text{max}}}(\Sob{\tht^{\be_0}(t)}{H^{s}}+\Sob{\overline{\tht}}{H^s})\,dt\right)\Sob{\tht_0}{H^{s+1}}\le C
\end{align*}
for $t \in [0,T^\be_{\text{max}}]$. As a result, $\tht^\be(T^\be_{\text{max}})$ is finite, leading to a contradiction since $T^\be_{\text{max}}$ is the maximal time of existence. Therefore, $T^\be_{\text{max}}\ge T_0$ as claimed. \qed
\section{Proof of \cref{main:T:2}}
We will now prove \cref{main:T:2}, which corresponds to the case $\be_0=1$ (the SQG equation). The proof mostly follows a similar approach to that of \cref{main:T:1} except for slight modifications in the estimates of $K_1, K_2, K_3$ in \eqref{energy:K}, $L_3$ in \eqref{energy:L}, and $M_2$, $M_3$ in \eqref{energy:H:M}. To be specific, we will now estimate these terms with bounds that remain valid as $\be \to 1$. The main ingredient is the application of \cref{commutator:estimate:1}, which allows a uniform bound for commutators that appear in the energy arguments. This, along with uniform bounds of \cref{lem:law:1} and \cref{lem:law:2} provides the requisite control of the nonlinearity for treating the case $\be_0=1$. To avoid redundancy, we only show the new estimates required.\par
Applying the Cauchy-Schwarz inequality followed by \cref{commutator:estimate:1}, \eqref{commutator:1:est:1} with $g=\nabla \overline{\tht}$, $f=\overline{\tht}$, and $p=2, q_1=\infty, r_1=2$, we obtain
\begin{align}\label{est:K1:T2}
    |K_1|&=|\lb [\nabla^\perp \Lam^{\be-2}\cdot, \nabla \overline{\tht}]\overline{\tht},\overline{\tht} \rb|\notag\\&\le C(\Sob{ \nabla^2 \overline{\tht}}{L^\infty}\Sob{\overline{\tht}}{L^2}+\Sob{\nabla \overline{\tht}}{L^\infty}\Sob{\overline{\tht}}{L^1})\Sob{\overline{\tht}}{L^2}\notag\\&
    \le C \Sob{\overline{\tht}}{H^s}\Sob{\overline{\tht}}{L^2}^2+C\Sob{\overline{\tht}}{H^s}\Sob{\overline{\tht}}{L^1}\Sob{\overline{\tht}}{L^2}.
\end{align}
Similarly, we obtain
\begin{align}\label{est:K2:T2}
    |K_2|\le C \Sob{\overline{\tht}}{H^s}\Sob{\tht^{\be_0}}{L^2}\Sob{\overline{\tht}}{L^2}+C\Sob{\overline{\tht}}{H^s}\Sob{{\tht^{\be_0}}}{L^1}\Sob{\overline{\tht}}{L^2},
\end{align}
and
\begin{align}\label{est:K3:T2}
    |K_3|\le  C \Sob{{\tht^{\be_0}}}{H^s}\Sob{\overline{\tht}}{L^2}^2+C\Sob{{\tht^{\be_0}}}{H^s}\Sob{\overline{\tht}}{L^1}\Sob{\overline{\tht}}{L^2}.
\end{align}
Applying the Cauchy-Schwarz inequality followed by \cref{lem:law:2} with $q=\overline{\tht}$, $\tht=\tht^{\be_0}$, and $\si=s$, we obtain
\begin{align}\label{est:L3:T2}
     |L_3|\le C\Sob{\tht^{\be_0}}{H^{s+1}}\Sob{\overline{\tht}}{H^s}^2+C\Sob{\tht^{\be_0}}{H^{s+1}}\Sob{\overline{\tht}}{L^1}\Sob{\overline{\tht}}{H^s}.
 \end{align}
 Now, we estimate $M_3$ in a different way. Recall that 
  \begin{align*}
     |M_3|\le (\be-1)\int_{|x-y|\ge 1}\frac{|\ln(|x-y|)|}{|x-y|^{1+\be}}|J^s \tht^{\be_0}(y)|\,dy.
 \end{align*}
Let $\kap \in (0,1)$. Making a change of variables and applying the Young's convolution inequality, we obtain
\begin{align}
     \Sob{M_3}{L^2}&\le (\be-1)\left(\int_{|z|\ge 1}\left(\frac{|\ln|z||}{|z|^{1+\be}}\right)^{q}\,dz\right)^{\frac{1}{q}}\Sob{J^s\tht^{\be_0}}{L^r}\notag\\
     &\le (\be-1)\left(\int_{|z|\ge 1}\frac{1}{|z|^{(1+\be-\kap)q}}\,dz\right)^{\frac{1}{q}}\Sob{J^s \tht^{\be_0}}{L^r}\notag\\
     &\le (\be-1)\frac{C}{((1+\be-\kap)q-2)^{\frac{1}{q}}}\Sob{J^s\tht^{\be_0}}{L^r}, \label{est:M3:T:2}
\end{align}
 where $q$ and $r$ satisfy
 \[\frac{1}{q}+\frac{1}{r}=\frac{3}{2},\quad q>\frac{2}{1+\be-\kap}.\]
 As a result, we have
 \[r<\frac{2}{2-\be+\kap}<2\]
 for $\be<1+\kap$.
 Using the Gagliardo-Nirenberg inequality \cite{BahouriCheminDanchinBook2011}, we have
 \begin{align*}
     \Sob{\Lam^s \tht^{\be_0}}{L^r}\le \Sob{\tht^{\be_0}}{L^1}^{1-\tau}\Sob{\Lam^{s+1}\tht^{\be_0}}{L^2}^{\tau},
 \end{align*}
 where  $\tau=1-\frac{2}{r(s+2)}$ and $r$ satisfies $r\ge \frac{2(s+1)}{(s+2)}$. For $\kap$  sufficiently small, we can choose $r$ such that
 \[\frac{2(s+1)}{(s+2)}\le r< \frac{2}{2-\be+\kap}.\]
Collecting the estimates in \eqref{est:M1}, \eqref{est:M2:case1}, \eqref{est:M2:case2}, and \eqref{est:M3:T:2}, setting $\veps=(\be-1)$, and then applying Young's inequality, we obtain
     \begin{align}\label{est:H:T:2}
     \Sob{\mathcal{H}^{\be,\be_0}(J^s\tht^{\be_0})}{L^2}\le C((\be-1)^{2-\be}+(\be-1)+(\be-1)|\ln (\be-1)|)\left(\Sob{\tht^{\be_0}}{H^{s+1}}+\Sob{\tht^{\be_0}}{L^1}\right).
 \end{align}
 Using the estimates \eqref{est:K1:T2}-\eqref{est:K3:T2}, \eqref{est:L1}, \eqref{est:L2}, \eqref{est:L3:T2}, and \eqref{est:H:T:2}, in \eqref{energy:K} and \eqref{energy:L}, and proceeding just like before, we obtain
 \begin{align*}
     \frac{d}{dt}\Sob{\overline{\tht}(t)}{H^s}\le& C\Sob{\overline{\tht}}{H^s}^2+C\Sob{\overline{\tht}}{H^s}(\Sob{\tht^{\be_0}}{H^{s+1}}+\Sob{\tht^\be}{L^1}+\Sob{\tht^{\be_0}}{L^1})\\
     &+C \zeta(\be)(\Sob{\tht^{\be_0}}{H^{s+1}}^2+\Sob{\tht^\be}{L^1}^2+\Sob{\tht^{\be_0}}{L^1}^2),
 \end{align*}
 where 
 \[\zeta(\be)=(\be-1)^{2-\be}+(\be-1)+(\be-1)|\ln (\be-1)|.\]
 As $\nabla \cdot u=0$, $\Sob{\tht}{L^p}$ is conserved for any $p\ge 1$. Therefore, $\Sob{\tht^\be}{L^1}+\Sob{\tht^{\be_0}}{L^1}$ is bounded if the initial data $\tht_0 \in L^1$. Proceeding just like in the proof of \cref{main:T:2}, we deduce that
 \begin{align*}
     \Sob{\overline{\tht}}{H^{s}}\le C\left((\be-1)^{2-\be}+(\be-1)|\ln (\be-1)|\right),
 \end{align*}
 thus completing the proof
 \qed

\bibliographystyle{plain}
\bibliography{main_bib.bib}
\vspace{.3in}
\end{document}